\documentclass[10pt,headsepline, a4paper,cleardoubleempty]{scrartcl}

\usepackage[utf8]{inputenc}
\usepackage[intlimits]{amsmath}
\usepackage{amsfonts}
\usepackage{amssymb}
\usepackage[amsmath, thmmarks]{ntheorem}
\usepackage{enumerate}
\usepackage{tikz}

\newcommand{\R}{\mathbb{R}}
\newcommand{\Z}{\mathbb{Z}}
\newcommand{\Q}{\mathbb{Q}}
\newcommand{\N}{\mathbb{N}}
\newcommand{\C}{\mathbb{C}}

\renewcommand{\sl}{\operatorname{SL}(2,\Z)}

\newcommand{\rank}{\operatorname{rank}}

\makeatother
\newtheorem{vorlage}{}[section]
\newtheorem{prop}[vorlage]{Proposition}
\newtheorem{lemma}[vorlage]{Lemma}
\newtheorem{cor}[vorlage]{Corollary}

\newtheorem{theorem}[vorlage]{Theorem}
\theorembodyfont{\rmfamily}
\newtheorem{definition}[vorlage]{Definition}

\newtheorem{remdef}[vorlage]{Remark and Definition}
\newtheorem{ex}[vorlage]{Example}
\theoremstyle{nonumberbreak}
\theoremsymbol{\ensuremath{\square}}
\newtheorem{proof}{Proof}
\newtheorem{mainproof}{Proof of Theorem 1.2}

\title{Modular forms for the $A_{1}$-tower}
\author{Martin Woitalla}

\begin{document}
\maketitle

\begin{abstract}
In the 1960's Igusa determined the graded ring of Siegel modular forms of genus two. He used theta series to construct $\chi_{5}$, the cusp form of lowest weight for the group $\operatorname{Sp}(2,\Z)$. In 2010 Gritsenko found three towers of orthogonal type modular forms which are connected with certain series of root lattices. In this setting Siegel modular forms can be identified with the orthogonal group of signature $(2,3)$ for the lattice $A_{1}$ and Igusa's form $\chi_{5}$ appears as the roof of this tower. We use this interpretation to construct a framework for this tower which uses three different types of constructions for modular forms. It turns out that our method produces simple coordinates. 
\end{abstract}
\section{Introduction}
Let $V$ be a real quadratic space of signature $(2,n)$ where $n\in\N_{\geq 3}$. The bilinear form of $V$ is denoted by $(\cdot,\cdot)$. The group of all isometries of $V$ is called the \textit{orthogonal group of} $V$ and is given by
\[\operatorname{O}(V)=\{g\in\operatorname{GL}(V)\,|\,\forall v\in V \,:\, (gv,gv)=(v,v)\}\,.\]
We extend the bilinear form to $V\otimes\C$ by $\C$-linearity. We consider
\[\mathcal{D}^{\pm}=\{[\mathcal{Z}]\in\mathbb{P}(V\otimes\C)\,|\,(\mathcal{Z},\mathcal{Z})=0\,,\,(\mathcal{Z},\overline{\mathcal{Z}})>0\}\]
on which $\operatorname{O}(V)$ acts as a linear group. The domain $\mathcal{D}^{\pm}$ has two connected components. We choose one of them and denote it by $\mathcal{D}$. We define the subgroups
\[\operatorname{O}(V)^{+}\,,\,\operatorname{SO}(V)^{+}=\{g\in\operatorname{O}(V)^{+}\,|\,\det(g)=1\}\]
of index $2$ and $4$, respectively, which fix $\mathcal{D}$. The latter group is the connected component of the identity and is well-known to be a semisimple and noncompact Lie group. Its maximal compact subgroup is given by $K=\operatorname{SO}(2)\times\operatorname{SO}(n)$ and the Hermitian symmetric space $\operatorname{SO}(V)^{+}/K$ is isomorphic to $\mathcal{D}$. The affine cone is defined as
\[\mathcal{D}^{\bullet}=\{\mathcal{Z}\in V\otimes \C\,|\,[\mathcal{Z}]\in \mathcal{D}\}\,.\]
Let $L\subseteq V$ be a positive definite even lattice such that the dimension of $L\otimes \R$ is $n-2$ and let 
\[U,U_{1}\cong \begin{pmatrix}
                                                                        0&1\\1&0
                                                                       \end{pmatrix}\]
 be two integral hyperbolic planes. Denote by $L(-1)$ the associated negative definite lattice. We consider the arithmetic subgroup
 \[\operatorname{O}(L_{2})^{+}=\{g\in\operatorname{O}(V)^{+}\,|\,g\,L_{2}\subseteq L_{2}\}\]
where $L_{2}\cong U\perp U_{1}\perp L(-1)$ is an even lattice in $V$. 
For any subgroup $\Gamma\leq \operatorname{O}(L_{2})^{+}$ of finite index we consider the modular variety $\Gamma\backslash\mathcal{D}$. This is a noncompact space. In \cite{BorJ} and  \cite{BaBor} the Satake-Baily-Borel compactification of this space is considered. The boundary components of this compactification are usually called the \textit{cusps} of $\Gamma\backslash\mathcal{D}$. In \cite{BaBor} the authors construct a general version of Siegel's $\Phi$-operator to assign boundary values to automorphic forms with respect to $\Gamma$. This is used in the following definition. 
\begin{definition}\label{def:modular form}
 Let $\Gamma$ be a subgroup of $\operatorname{O}(L_{2})^{+}$. A \textit{modular form} of weight $k\in \Z$ and character $\chi\,:\,\Gamma\to\C^{\times}$ with respect to $\Gamma$ is a holomorphic function  $F\,:\,\mathcal{D}^{\bullet}\to\C$ such that 
\[\begin{aligned}
   &F(t\mathcal{Z})=t^{-k}F(\mathcal{Z})\quad\text{ for all }t\in\C^{\times}\,,&\\
   &F(g\mathcal{Z})=\chi(g)F(\mathcal{Z})\quad\text{ for all }g\in\Gamma\,.&
  \end{aligned}
\]
A modular form is called a \textit{cusp form} if it vanishes at every cusp. The space of modular forms of weight $k$ and character $\chi$ for the group $\Gamma$ will be denoted by $\mathcal{M}_{k}(\Gamma,\chi)$. For the subspace of cusp forms we will write $\mathcal{S}_{k}(\Gamma,\chi)$.  
\end{definition} 
Let $\Gamma\leq \operatorname{O}(L_{2})^{+}$ be a subgroup of finite index and denote by $\Gamma'=[\Gamma,\Gamma]$ the commutator subgroup of $\Gamma$. We denote by 
\[\mathcal{A}(\Gamma')=\bigoplus_{k=0}^{\infty}{\mathcal{M}_{k}(\Gamma',1)}\]
the graded ring of modular forms.
It is well-known that this ring is finitely generated. In the sequel the notation $F_{k}\in\mathcal{A}(\Gamma')$ means that $F$ is a homogeneous modular form of weight $k$. We define the \textit{dual lattice} of $L$ as the $\Z$-module
 \[L^{\vee}\index{$L^{\vee}$}:=\{x\in V\,|\,\forall l\in L:\,(x,l)\in\Z\}\,.\]
 Since $L$ is even we have $L\subseteq L^{\vee}$. We define the \textit{discriminant group} as the finite abelian group 
 \[D(L):=L^{\vee}/L\,.\]
The group $\operatorname{O}(L_{2})^{+}$ acts on the discriminant group $D(L_{2})$. The kernel of this action is denoted by $\widetilde{\operatorname{O}}(L_{2})^{+}$. This subgroup will be interesting for our further considerations.  Another natural subgroup is the finite group $\operatorname{O}(L)$ which consists of all automorphisms of the positive definite lattice $L$.  
\vspace{5mm}

\noindent We put our focus to a special series of lattices. Denote by $(\cdot,\cdot)_{m}$ the standard scalar product on $\R^{m}$. If $\varepsilon_{1},\dots,\varepsilon_{m}$ denotes the standard basis of $\R^{m}$ we consider the following $\Z$-module of rank $m$ 
\[mA_{1}=\langle \varepsilon_{1},\dots, \varepsilon_{m}\rangle_{\Z}\,.\]
If we equip $mA_{1}$ with the bilinear form $2(\cdot,\cdot)_{m}$ we obtain a series of (reducible) root lattices where $mA_{1}$ should be understood as an $m$-fold perpendicular sum of type $A_{1}$ root lattices. Due to some low-dimensional exceptional isogenies this series has connections to modular varietes of unitary and symplectic type.
\begin{description}\label{correspondences to classical Siegel,Hermitian and Quaternary varieties}
 \item[Case $m=1$.] In this case $L_{2}(A_{1})$ has signature $(2,3)$ and the group 
\[\Gamma=\operatorname{O}(L_{2}(A_{1}))^{+}\cap\operatorname{SO}(V)^{+}\]
is isomorphic to the projective symplectic group $\operatorname{PSp}(2,\Z)$, compare \cite[Proposition 1.2]{GH}. The variety $\Gamma\backslash\mathcal{D}$ has been studied by Igusa in \cite{Ig}. He showed that the graded ring of Siegel modular forms $\mathcal{A}(\Gamma')$ of genus two is generated by the Siegel Eisenstein series $E_{4},E_{6}$ and cusp forms $\chi_{5},\psi_{12}$ and $\chi_{30}$. 
For any modular form $F\in \mathcal{M}_{k}(\operatorname{O}(L_{2}(A_{1}))^{+},\det^{\kappa})$ where $\kappa,k\in\N_{0}$ the modularity conditions yield
\[(-1)^{\kappa}F(\mathcal{Z})= F((-I_{5}\mathcal{Z}))=F(-\mathcal{Z})=(-1)^{k}F(\mathcal{Z})\,.\]
Hence the determinant-character corresponds to the weight parity in the symplectic setting. According to Igusa's result we have 
\begin{equation}\label{Igusas result}
 \bigoplus_{k\in\Z}{\mathcal{M}_{k}(\operatorname{O}(L_{2}(A_{1}))^{+},1)}\cong\C[E_{4},E_{6},\chi_{5}^{2}, \psi_{12}]\,.
\end{equation}
\item[Case $m=2$.] We consider the Gaussian number field $K=\Q(\sqrt{-1})$ whose ring of integers equals $\mathfrak{o}_{K}=\Z+\Z\sqrt{-1}$. The \textit{special unitary} group $\operatorname{SU}(\mathfrak{o}_{K})\subseteq \operatorname{SL}(4,\mathfrak{o}_{K})$ acts on the Hermitian half-plane of degree two. This can be used to show that $\widetilde{\operatorname{SO}}(L_{2}(2A_{1}))^{+}/\{\pm I_{6}\}$ is isomorphic to $\operatorname{SU}(\mathfrak{o}_{K})/\{\pm I_{4}\}$, compare \cite[Remark 3.3.4]{Wo}. This case has been investigated by Freitag in \cite{F2} and later by Dern and Krieg in \cite{DK}.
\item[Case $m=4$.] Let $Q$ be the rational quaternion algebra of signature $(-1,-1)$. As a vector space over $\Q$ we have
\[Q=\Q+\Q i+\Q j+ \Q ij\,,\,i^{2}=j^{2}=-1\,.\]
A maximal order in $Q$ is given by $\mathfrak{o}=\Z+\Z i+\Z j+\Z \omega$ where $\omega=\frac{1}{2}(1+i+j+ij)$. The order $\mathfrak{o}_{0}=\Z+\Z i +\Z j +\Z ij$ is a sublattice of index $2$ and is isomorphic to $4A_{1}$. This lattice is also known as the ring of \textit{Lipschitz quaternions} and $\mathfrak{o}$ is the ring of \textit{Hurwitz quaternions}. The corresponding modular group is $\operatorname{Sp}(2,\mathfrak{o}_{0})$ and can be identified with a subgroup of $\operatorname{O}(L_{2}(4A_{1}))/\{\pm I_{8}\}$. The rings of quaternionic modular forms have been investigated by Freitag and Krieg in \cite{FH}, \cite{K2} and \cite{K4}.
\end{description}
In \cite{G1} Gritsenko found three towers of reflective modular forms. In his construction  Igusa's modular form $\chi_{5}$ is the roof of the $4A_{1}$-tower. In the sequel we will develop a framework around Gritsenko's tower without making use of the exceptional isogenies. We will use three different types of coordinates. 
\begin{enumerate}[(i)]
 \item The so called \textit{Eisenstein type} modular forms constitute the first type. These forms are pullbacks of Gritsenko's singular modular form for the even unimodular lattice of signature $(2,10)$. If we additionally take into account the heat operator for several variables considered in \cite{ChoKi} we obtain non-cusp forms of weight $4$ and $6$. The common source function for all these forms is the classical Eisenstein series of weight $4$ for the group $\sl$. 
\item The second family of modular forms arises as a natural extension of the $4A_{1}$-tower of reflective modular forms. These forms are called \textit{theta type} modular forms and are investigated in \cite{Wo1}. The source function of this tower is $\Delta_{12}$,  the first cusp form for the group $\sl$. 
 \item  The third family is of \textit{baby monster type} (bm type) and arises as a quasi-pullback of Borcherds famous $\Phi_{12}$-function which is the denominator function of the fake monster Lie algebra, compare \cite{Bo2} and \cite{GHS2}. In \cite{Gr} an algorithm is presented to produce many reflective modular forms of baby monster type. We can again consider $\Delta_{12}$ as the common source function of the bm type modular forms. 
\end{enumerate}
Besides the determinant-character the group $\operatorname{O}(L_{2}(mA_{1}))^{+}$ admits two more finite characters. The discriminant group $D(mA_{1})$ is isomorphic to $m$ copies of the cyclic group of order two. The quadratic form on $L_{2}$ induces the \textit{discriminant form} on $D(L_{2})$ and we obtain the finite orthogonal group $\operatorname{O}(D(L_{2}))$ as the image of the natural homomorphism
\[ \pi\,:\,\operatorname{O}(L_{2})^{+}\to\operatorname{O}(D(L_{2}))\,.\]
The kernel of this homomorphism is the stable orthogonal group $\widetilde{\operatorname{O}}(L_{2})^{+}$. In our case $\operatorname{O}(D(L_{2}(mA_{1})))\cong \operatorname{O}(D(mA_{1}))$ is isomorphic to the symmetric group on $m$ letters $\mathcal{S}_{m}$ and $\pi$ is surjective. This yields a binary character
\[v_{\pi}\,:\,\operatorname{O}(L_{2}(mA_{1}))^{+}\twoheadrightarrow \mathcal{S}_{m}\xrightarrow[]{\operatorname{sgn}}\{\pm 1\}\,.\]
The construction of $mA_{1}$ implies that $(x,y)_{mA_{1}}\in 2\Z$ for all $x,y\in mA_{1}$. In this case we can construct another binary character, see e.g. \cite[Proposition 1.26]{Kl} and \cite[Theorem 2.2]{CG}:
\[v_{2}\,:\,\operatorname{O}(L_{2}(mA_{1}))^{+}\to\operatorname{Sp}(2,\mathbb{F}_{2})\to\mathcal{S}_{6}\xrightarrow[]{\operatorname{sgn}}\{\pm 1\}\,.\]
The construction of $v_{2}$ implies 
\[\ker v_{2}\,\cap\, \widetilde{\operatorname{SO}}(L_{2}(mA_{1}))^{+}\lneq \widetilde{\operatorname{SO}}(L_{2}(mA_{1}))^{+}\quad,\quad \operatorname{O}(mA_{1})\leq \ker v_{2}\,.\] 
We set for abbreviation $\Gamma_{m}:=\operatorname{O}(L_{2}(mA_{1}))^{+}$ and $\widetilde{\Gamma}_{m}:=\widetilde{\operatorname{O}}(L_{2}(mA_{1}))^{+}$. In \cite[Proposition 5.4.2]{Wo} it is shown that $\Gamma_{m}/\Gamma_{m}'\cong\langle \det, v_{2},v_{\pi}\rangle\cong \mathcal{C}_{2}^{3}$ if $m=2,3,4$ and $\Gamma_{1}/\Gamma_{1}'\cong\langle \det, v_{2}\rangle.$ The paper is organized in the following way.
\vspace{5mm}

\noindent In section 2 we introduce Jacobi forms of theta type. These forms are obtained by twisting powers of Jacobi's theta function of weight and index $1/2$ with the weak Jacobi form of weight 0 and index 1 defined in \cite{EZ} and a multiplication with suitable powers of Dedekind's eta function. Moreover Jacobi forms of Eisenstein type are introduced. The arithmetic lifting of these functions yields modular forms for the orthogonal group with trivial character. In section 3 we consider two refinements of theta type Jacobi forms which yield two more series of modular forms with respect to binary characters. The first one uses a variant of the arithmetic lifting for Jacobi forms of half-integral index given in \cite{CG}. The second series is obtained by considering a cusp form of weight $24$ for the lattice $D_{4}$. We rewrite the coordinates of this function for the sublattice $4A_{1}$ and obtain a series of length three by considering quasi-pullbacks. Finally the quasi-pullbacks of Borcherd's function $\Phi_{12}$ produce another series of modular forms including Igusa's function $\chi_{30}$. This enables us to state our main theorem. 
\begin{theorem}\label{theorem:sructure graded rings with respect to character A1}
Let $m\in \{1,2,3,4\}$. The graded ring of modular forms $\mathcal{A}(\Gamma_{m}')$ is generated by the $m$-th row of the following table
 \begin{center}
 \begin{tikzpicture}[scale=0.3,node distance=2cm,auto]
   \node (12+A1) at (0cm,-2.5cm) {$\textcolor{black}{F_{12}^{A_{1}}}$};
   \node (12+2A1) at (0cm,-6cm) {$\textcolor{black}{F_{12}^{2A_{1}}}$};
   \node (12+3A1) at (0cm,-9.5cm) {$\textcolor{black}{F_{12}^{3A_{1}}}$};
   \node (12+4A1) at (0cm,-13cm) {$F_{12}^{4A_{1}}$};
   \node[] (10+A1) at (4cm,-2.5cm) {$\textcolor{black}{\chi_{5}^{A_{1}}}$};
   \node[] (10+2A1) at (4cm,-6cm) {$\textcolor{black}{F_{10}^{2A_{1}}}$};
   \node[] (10+3A1) at (4cm,-9.5cm) {$\textcolor{black}{F_{10}^{3A_{1}}}$};
   \node[] (10+4A1) at (4cm,-13cm) {$F_{10}^{4A_{1}}$};
   \node[] (8+2A1) at (8cm,-6cm) {$\textcolor{black}{\chi_{4}^{2A_{1}}}$};
   \node[] (8+3A1) at (8cm,-9.5cm) {$\textcolor{black}{F_{8}^{3A_{1}}}$};
   \node[] (8+4A1) at (8cm,-13	cm) {$F_{8}^{4A_{1}}$};
   \node[] (6+3A1) at (12cm,-9.5cm) {$\textcolor{black}{\chi_{3}^{3A_{1}}}$};
  \node[] (6+4A1) at (12cm,-13cm) {$F_{6}^{4A_{1}}$};
  \node[] (4+4A1) at (16cm,-13cm) {$\chi_{2}^{4A_{1}}$};
  \node[] (e4+4A1) at (-9cm,-13cm) {$\mathcal{E}_{4}^{4A_{1}}$};
  \node[] (e4+3A1) at (-9cm,-9.5cm) {$\mathcal{E}_{4}^{3A_{1}}$};
  \node[] (e4+2A1) at (-9cm,-6cm) {$\mathcal{E}_{4}^{2A_{1}}$};
  \node[] (e4+A1) at (-9cm,-2.5cm) {$\mathcal{E}_{4}^{A_{1}}$};
  \node[] (e6+4A1) at (-4cm,-13cm) {$\mathcal{E}_{6}^{4A_{1}}$};
 \node[] (e6+3A1) at (-4cm,-9.5cm) {$\mathcal{E}_{6}^{3A_{1}}$};
 \node[] (e6+2A1) at (-4cm,-6cm) {$\mathcal{E}_{6}^{2A_{1}}$};
 \node[] (e6+A1) at (-4cm,-2.5cm) {$\mathcal{E}_{6}^{A_{1}}$};
  \node[] (F30+A1) at (-16cm,-2.5cm) {$H_{30}^{A_{1}}$};
  \node[] (F30+2A1) at (-16cm,-6cm) {$H_{30}^{2A_{1}}$};
  \node[] (F30+3A1) at (-16cm,-9.5cm) {$H_{30}^{3A_{1}}$};
  \node[] (F30+4A1) at (-16cm,-13cm) {$H_{30}^{4A_{1}}$};
  \node[] (Delta10+2A1) at (12cm,-6cm) {$\Delta_{10}^{2A_{1}}$};
  \node[] (Delta18+3A1) at (16cm,-9.5cm) {$\Delta_{18}^{3A_{1}}$};
  \node[] (Delta24+4A1) at (20cm,-13cm) {$\Delta_{24}^{4A_{1}}$};
  \draw[] (-2cm,0cm) -- (-2cm, -15cm);
  \draw[] (-12cm,0cm) -- (-12cm, -15cm);
  \draw[] (-19cm,0cm) -- (-19cm, -15cm);
  \draw[] (-21cm,-1.5cm) -- (20cm, -1.5cm);
 \node[] at (-20cm,-0.5cm) {$m$};
 \node[] at (-16cm,-0.5cm) {\textnormal{bm type}};
 \node[] at (-7cm,-0.5cm) {\textnormal{Eisenstein type}};
 \node[] at (2cm,-0.5cm) {\textnormal{theta type}};
 \node[] at (-20cm,-2.5cm) {$1$};
 \node[] at (-20cm,-6.0cm) {$2$};
 \node[] at (-20cm,-9.5cm) {$3$};
 \node[] at (-20cm,-13.0cm) {$4$};
\end{tikzpicture}
\end{center}
where the index indicates the weight. 
\end{theorem}
The pullback structure underlying the above table is explained in section 2 and 3.
The generators in the cases $m=1,2,3$ have been determined before by Igusa, Freitag, Dern, Krieg and Klöcker whereas generators in the case $m=4$ have only been determined for the ring $\mathcal{A}(\Gamma_{4})$ in \cite{K4} best to the author's knowledge. Finally in section 4 we give a number theoretical application. The construction of Eisenstein type modular forms allows us to express some numbers of lattice points lying on a sphere as special values of $L$-functions which appear in \cite{BrK}.
\vspace{5mm}

\noindent\textbf{Acknowledgements} The results of this paper are part of the author's phd-thesis. The author would like to thank the supervisors Valery Gritsenko and Aloys Krieg for their guidance and support.

\section{Jacobi forms}
Let $L$ be a positive definite even lattice. The Jacobi group $\Gamma^{J}(L)$ is considered in \cite{CG} and is isomorphic to $\sl\ltimes H(L)$ where $H(L)$ denotes the integral Heisenberg group for the lattice $L$. We denote by $\mathbb{H}$ the upper half-plane in $\C$. Following \cite{EZ} and \cite{G} we define an action of the Jacobi group on the space of holomorphic functions defined on $\mathbb{H}\times (L\otimes \C)$. By considering the generators of the Jacobi group we can use this action to introduce the notion of a Jacobi form.
\begin{definition}\label{Jacobi form}
 Let $k,t\in\N_{0}$. A holomorphic function $\varphi\,:\,\mathbb{H}\times(L\otimes \C)\to\C$ is called a \textit{weak Jacobi form} of weight $k$ and index $t$ with character $\chi$ if the following conditions are satisfied:
\begin{enumerate}[(i)]
 \item For all $A=\begin{pmatrix}                                                                                                                                                                                                       a&b\\c&d                                                                                                                                                                                                             \end{pmatrix}\in\sl$:
\[\begin{aligned}
&\varphi\left(\frac{a\tau+b}{c\tau+d},\frac{\mathfrak{z}}{c\tau +d}\right)=\chi(A)(c\tau +d)^{k}e^{\pi i t\frac{c(\mathfrak{z},\mathfrak{z})}{c\tau+d}}\varphi(\tau,\mathfrak{z})\,.&\end{aligned}\]
\item For all $x,y\in L$:
\[\begin{aligned}&\varphi(\tau,\mathfrak{z}+x\tau+y)&&=\chi([x,y:-(x,y)/2])\cdot e^{-2\pi it(\frac{1}{2}(x,x)\tau+(x,\mathfrak{z}))}\,\varphi(\tau,\mathfrak{z})&
  \end{aligned}\]
where $[x,y:-(x,y)/2]\in H(L)$, compare \cite{CG} for the realization of $H(L)$.
\item The Fourier expansion of $\varphi$ has the shape
\[ \varphi(\tau,\mathfrak{z})=\sum_{\substack{n\in\N_{0}\\l\in\frac{1}{2}L^{\vee}}}{f(n,l)e^{2\pi i (n\tau+(l,\mathfrak{z}))}}\,.\]
\end{enumerate}
We call $\varphi$ a \textit{holomorphic Jacobi form} if the Fourier expansion ranges over all $n,l$ such that $2nt-(l,l)\geq 0$ and $\varphi$ is called  a \textit{Jacobi cusp form} if it ranges over all $n,l$ satisfying $2nt-(l,l)> 0$.
\end{definition}
\begin{remdef}
\begin{enumerate}[(a)]
 \item The action can be extended for $k,t\in\frac{1}{2}\Z_{\geq 0}$ and $\chi|_{\sl}$ being a multiplier system for $\sl$. Here we have to replace $\sl$ by the metaplectic cover $\operatorname{Mp}(2,\Z)$, see e.g. \cite{Br}. In this more general situation we use the notation
\[J^{(\textit{cusp})}_{k,L;t}(\chi)\subseteq J_{k,L;t}(\chi)\subseteq J^{(\textit{weak})}_{k,L;t}(\chi)\]
 for the corresponding spaces of Jacobi forms. If $\chi=1$ we write $J^{(\ast)}_{k,L;t}$ for each of these spaces. 
 \item The notion of a Jacobi form is compatible with Definition \ref{def:modular form}. To see this we note that we have an affine model for the homogeneous domain $\mathcal{D}$ given by
\[\mathcal{H}(L_{2})=
\left\{(\omega,\mathfrak{z},\tau)\in \C\times (L\otimes \C)\times \C\,\left|\begin{aligned}
                                                                                                  &\omega_{i},\tau_{i}>0\,,&\\&2\omega_{i}\tau_{i}-(\mathfrak{z}_{i},\mathfrak{z}_{i})>0
                                                                                                 \end{aligned}\right.\right\}
\]
where we have used the abbreviations
\[\omega_{i}:=\operatorname{Im}(\omega)\quad,\quad \tau_{i}:=\operatorname{Im}(\tau)\quad,\quad \mathfrak{z}_{i}:=\operatorname{Im}(\mathfrak{z})\,.\]
 Let $\varphi\in J_{k,L;t}(\chi)$ where we assume $k\in\Z$. We define a holomorphic function  on $\mathcal{H}(L_{2})$ by
 \[\widetilde{\varphi}(\tau,\mathfrak{z})=\varphi(\tau,\mathfrak{z})e^{2\pi it\omega}\,.\]
Since  $\mathcal{D}$ and  $\mathcal{H}(L_{2})$ are biholomorphically equivalent we can interpret $\widetilde{\varphi}$ as an element in $\mathcal{M}_{k}(\Gamma^{J}(L),\chi)$.  
\end{enumerate}
\end{remdef}
The following two examples are the basic ingredients to define theta type Jacobi forms.
\begin{ex}\label{eta function and Gritsenkos theta function}
 \begin{enumerate}[(a)]
  \item Dedekind's eta function is a Jacobi form of weight $1/2$ and index $0$ for every positive definite even lattice $L$, thus $\eta\in J_{1/2,L;0}^{\textit{(cusp)}}(v_{\eta})$ where $v_{\eta}$ is a multiplier system for $\sl$. 
  \item The Jacobi theta series of characteristic $(\frac{1}{2},\frac{1}{2})$ is given as 
\[\vartheta(\tau,z)\index{$\vartheta(\tau,z)$}=\sum_{n\in\Z}{\left( \frac{-4}{n}\right)q^{\frac{n^{2}}{8}}r^{\frac{n}{2}}}=\sum_{n\in\Z\,,\,n \equiv 1\bmod 2}{(-1)^{\frac{n-1}{2}}\exp\left(\frac{\pi i n^{2}\tau}{4}+\pi i n z\right)}\]
where $q=e^{2\pi i \tau}\,,\,\tau\in\mathbb{H}$ and $r=e^{2\pi i z}\,,\,z\in\C$. 
This function was originally discovered by Carl Gustav Jacob Jacobi. In \cite{GN} the authors reinterpreted this function as a modular form of half-integral weight and index. 
Jacobi's triple identity yields
\[\vartheta(\tau,z)=-q^{1/8}r^{-1/2}\prod_{n\geq 1}{(1-q^{n-1}r)(1-q^{n}r^{-1})(1-q^{n})}\,.\]
The function has the properties
\[\begin{aligned}
   &\vartheta(\tau,-z)&&=&&-\vartheta(\tau,z)\,,&\\&\vartheta(\tau,z+x\tau+y)&&=&&(-1)^{x+y}\exp(-\pi i(x^{2}\tau+2xz))\,\vartheta(\tau,z)&
  \end{aligned}\]
for all $x,y\in\Z$
and the set of zeroes of $\vartheta$ equals
\[\{x\tau+y\,|\,x,y\in\Z\}\,.\]  
 \end{enumerate}
\end{ex}
In the sequel let $M$ be a positive definite even lattice and $L\leq M$ be a sublattice of $M$. We define
\[L_{M}^{\perp}:=\{m\in M\,|\, \forall l\in L\,:\, (l,m)=0\}\]
and note that this is again a positive definite sublattice in $M$. The direct sum \[L\oplus L_{M}^{\perp}\leq M\] is a sublattice of finite index. 
 The next Lemma can be found in \cite{CG}, Proposition 3.1.
\begin{lemma}
 Let $L\leq M$ be a sublattice such that $\rank L < \rank M$ and let $\varphi\in J_{k,M;t}(\chi)$ be a Jacobi form of weight $k$ and index $t$ for the character $\chi$. Consider the decomposition $\mathfrak{z}_{M}=\mathfrak{z}_{L}\oplus\mathfrak{z}_{L^{\perp}}\in M\otimes\C=(L\oplus L_{M}^{\perp})\otimes \C$. We define the pullback of $\varphi$ to $L$ as the function  $\varphi\downharpoonright_{L}$ on $\mathbb{H}\times
 (L\otimes\C)$
\[\varphi\downharpoonright_{L}(\tau,\mathfrak{z}_{L}):=\varphi(\tau,\mathfrak{z}_{L}\oplus 0)\,.\]
Then $\varphi\downharpoonright_{L}\in  J_{k,L;t}(\left.\chi\right|_{\Gamma^{J}(L)})$ and the pullback maps cusp forms to cusp forms.
\end{lemma}
\begin{definition}\label{def:pullbak notation}
 Let $L\leq M$ be a sublattice of $M$ such that $\rank L < \rank M$ and $\varphi\in J_{k,M;t}^{\textit{(weak)}}$ be a Jacobi form of weight $k$ and index $t$.  Let $\psi\in J_{k,L;t}^{\textit{(weak)}}$. We say that $\psi$ is a \textit{pullback of} $\varphi$ if there exists some $\alpha\in \C^{\times}$ such that 
\(\psi=\alpha \cdot\varphi\downharpoonright_{L}\,.\)
In this case we use the notation $\varphi\to \psi$\,. We set $\varphi\downharpoonright_{L}:=\varphi$ if $\rank L =\rank M$.
\end{definition}
 We define
\[\vartheta_{L}\,:\,\mathbb{H}\times(L\otimes\C)\to\C\quad,\quad \vartheta_{L}(\tau,\mathfrak{z}):=\prod_{j=1}^{m}{\vartheta(\tau,(\mathfrak{z},\varepsilon_{j}))}\,.\]
This leads us to the notion of theta type Jacobi forms.
\begin{definition}\label{def:theta type JF}
 Let $L\subseteq \R^{m}$ be a positive definite even lattice and $\varphi\in J_{k,L;t}$. We say that $\varphi$ is of \textit{theta type} if there exists a sublattice $L'\subseteq L$, $\alpha\in\C^{\times}$ and integers $a,b\in\Z_{\geq 0}$ such that
\[\varphi\downharpoonright_{L'} (\tau,\mathfrak{z}')=\alpha\cdot\eta(\tau)^{a}\,\vartheta_{L'}(\tau,\mathfrak{z}')^{b}\,.\] 
Note that $\Delta_{12}(\tau)=\eta(\tau)^{24}$ is of theta type because $J_{k,L;0}$ is isomorphic to the space of weight $k$ modular forms for the group $\sl$.
\end{definition}
 For $\mathfrak{z}\in (mA_{1})\otimes\C$ we write 
\begin{equation}\label{choice of coordinates for mA1}
\mathfrak{z}=\sum_{j=1}^{m}{z_{j}\varepsilon_{j}}=:(z_{1},\dots,z_{m})\quad,\quad z_{j}\in\C\,. 
\end{equation}
For the next Proposition we note that
\[\vartheta_{4A_{1}}(\tau,\mathfrak{z})=\vartheta(\tau,z_{1})\vartheta(\tau,z_{2})\vartheta(\tau,z_{3})\vartheta(\tau,z_{4})\in J_{2,4A_{1};1/2}(\chi_{2})\]
has already been constructed in \cite{G1} where $\chi_{2}$ is a binary character. In the next Proposition the square of this function is denoted by $\psi_{4,4A_{1}}$. In \cite[Proposition 3.7]{Wo1} the author constructs a tower of theta type Jacobi forms for the lattice $4A_{1}$.
\begin{prop}\label{prop:tower for 4A1}
There exists the following diagram of theta type Jacobi forms for the $A_{1}$-tower
 \begin{center}
 \begin{tikzpicture}[scale=0.5,node distance=2cm,auto]
  \node[] (delta) at (0cm,0cm) {$\textcolor{black}{\Delta_{12}}$};
   \node (12+A1) at (0cm,-2.5cm) {$\textcolor{black}{\varphi_{12,A_{1}}}$};
   \node (12+2A1) at (0cm,-5cm) {$\textcolor{black}{\varphi_{12,2A_{1}}}$};
   \node (12+3A1) at (0cm,-7.5cm) {$\textcolor{black}{\varphi_{12,3A_{1}}}$};
   \node (12+4A1) at (0cm,-10cm) {$\varphi_{12,4A_{1}}$};
   \node[] (10+A1) at (5cm,-2.5cm) {$\textcolor{black}{\psi_{10,A_{1}}}$};
   \node[] (10+2A1) at (5cm,-5cm) {$\textcolor{black}{\varphi_{10,2A_{1}}}$};
   \node[] (10+3A1) at (5cm,-7.5cm) {$\textcolor{black}{\varphi_{10,3A_{1}}}$};
   \node[] (10+4A1) at (5cm,-10cm) {$\varphi_{10,4A_{1}}$};
   \node[] (8+2A1) at (10cm,-5cm) {$\textcolor{black}{\psi_{8,2A_{1}}}$};
   \node[] (8+3A1) at (10cm,-7.5cm) {$\textcolor{black}{\varphi_{8,3A_{1}}}$};
   \node[] (8+4A1) at (10cm,-10cm) {$\varphi_{8,4A_{1}}$};
   \node[] (6+3A1) at (15cm,-7.5cm) {$\textcolor{black}{\psi_{6,3A_{1}}}$};
  \node[] (6+4A1) at (15cm,-10cm) {$\varphi_{6,4A_{1}}$};
  \node[] (4+4A1) at (20cm,-10cm) {$\psi_{4,4A_{1}}$};
  \node[] () at (19.0cm,-8.00cm) {$\left.\frac{\partial^{2}}{\partial z_{4}^{2}}\right|_{z_{4}=0}$};
  \node[] () at (14.0cm,-5.50cm) {$\left.\frac{\partial^{2}}{\partial z_{3}^{2}}\right|_{z_{3}=0}$};
  \node[] () at (9.0cm,-3.00cm) {$\left.\frac{\partial^{2}}{\partial z_{2}^{2}}\right|_{z_{2}=0}$};
  \node[] () at (4.0cm,-0.50cm) {$\left.\frac{\partial^{2}}{\partial z_{1}^{2}}\right|_{z_{1}=0}$};
\draw[<-](delta) to node {} (12+A1);
\draw[<-](12+A1) to node {} (12+2A1);
\draw[<-](12+2A1) to node {} (12+3A1);
\draw[<-](12+3A1) to node {} (12+4A1);
 \draw[<-,dashed](delta) to node {} (10+A1);
\draw[<-](10+A1) to node {} (10+2A1);
\draw[<-](10+2A1) to node {} (10+3A1);
\draw[<-](10+3A1) to node {} (10+4A1);
 \draw[<-,dashed](10+A1) to node {} (8+2A1);
\draw[<-](8+2A1) to node {} (8+3A1);
\draw[<-](8+3A1) to node {} (8+4A1);
 \draw[<-,dashed](8+2A1) to node {} (6+3A1);
 \draw[<-,dashed](6+3A1) to node {} (4+4A1);
\draw[<-](6+3A1) to node {} (6+4A1);
\end{tikzpicture} 
\end{center}
where $\psi_{k,mA_{1}},\varphi_{k,mA_{1}}\in J_{k,mA_{1};1}$. Except for the last line all forms are cusp forms. 
\end{prop}
We recall that there exists a unique (up to isomorphism) positive definite even lattice $E_{8}$ in dimension $8$ which is unimodular. 
Following \cite{G} we can attach a Jacobi theta series to $E_{8}$:
\[\Theta_{E_{8}}(\tau,\mathfrak{z})=\sum_{l\in E_{8}}{\exp(\pi i (l,l)\tau+2\pi i(l,\mathfrak{z}))}\,.\]
Then $\Theta_{E_{8}}\in J_{4,E_{8};1}$ is a singular Jacobi form for $E_{8}$.   We fix a chain of embeddings 
\[A_{1}\hookrightarrow 2A_{1}\hookrightarrow 3A_{1}\hookrightarrow 4A_{1}\hookrightarrow E_{8}\,.\]
We will investigate the pullbacks
\[\epsilon_{4,mA_{1}}\index{$\epsilon_{4,mA_{1}}$}=\Theta_{E_{8}}\downharpoonright_{mA_{1}}\in J_{4,mA_{1};1}\,.\]
\begin{prop}\label{prop:modularity of the E8 pullbacks for A1 tower}
 Let $\sigma\in \operatorname{O}(mA_{1})$ for $m\in\{1,2,3, 4\}$. Then $\sigma$ can be extended to $\operatorname{O}(E_{8})$ and for any sublattice $L\leq E_{8}$ where $L\cong mA_{1}$ there exists some $g\in\operatorname{O}(E_{8})$ such that $g.L=mA_{1}$. Moreover $\epsilon_{4,mA_{1}}$ is invariant under the transformation induced by $\sigma$ such that for all $\tau,\mathfrak{z}$ one has \[\epsilon_{4,mA_{1}}(\tau,\sigma.\mathfrak{z})=\epsilon_{4,mA_{1}}(\tau,\mathfrak{z})\,.\]
\end{prop}
\begin{proof}
We denote by $K_{m}:=(mA_{1})_{E_{8}}^{\perp}$ the orthogonal complement of $mA_{1}$ in $E_{8}$. For the proof of the statement we will have to investigate the discriminant form 
\[q: D(L)\to\Q/2\Z\quad,\quad x +L \mapsto (x,x) +2\Z\,.\] 
In the following list one can find a root system which is isomorphic to $K_{m}$
\begin{center}
\begin{tabular}{l|cccc}
$m$ & 1 & 2 & 3 & 4\\\hline
$K_{m}$ & $E_{7}$ & $D_{6}$ & $A_{1}\oplus D_{4}$ & $4A_{1}$
\end{tabular}\,.
\end{center}
From \cite[Proposition 1.6.1]{N} we know that $\sigma$ can be extended to $\operatorname{O}(E_{8})$ 
\begin{equation}\label{Nikulins condition}
 \text{if the natural homomorphism }\operatorname{O}(K_{m})\to \operatorname{O}(D(K_{m})) \text{ is surjective.}
\end{equation}
According to \cite[Chapter 4, Section 8.2]{CS} we have $\operatorname{O}(D(E_{7}))=\{\operatorname{id}\}$ because $E_{7}^{\vee}=E_{7}\cup (v+E_{7})$ where $q(v+E_{7})=\frac{3}{2}+2\Z$
which grants the surjectivity in this case.
The lattice $D_{m},m\geq 3$ can be realized as the $\Z$-module with basis
\[\varepsilon_{2}+\varepsilon_{1},\varepsilon_{2}-\varepsilon_{1},\varepsilon_{3}-\varepsilon_{2},\dots, \varepsilon_{m}-\varepsilon_{m-1}\,.\]
Moreover this lattice can be described as the following subset in  $\Z^{m}$:
\[D_{m}=\{x\in\Z^{m}\,|\,x_{1}+\dots+x_{m} =  0 \bmod 2\}\]
We define $w:=\frac{1}{2}(\varepsilon_{1}+\dots+\varepsilon_{m})\in D_{m}^{\vee}$. The values of the discriminant form for the representatives of $D(D_{m})$ are given as follows 
\begin{center}
\begin{tabular}{l|llll}
$l$ & $0$ & $\varepsilon_{1}$ & $w$ & $\varepsilon_{1}+w$\\\hline
$q(l)$ & $0+2\Z$ & $1+2\Z$ & $\frac{m}{4}+2\Z$ & $\frac{m}{4}+2\Z$
\end{tabular}
\end{center}
If $m\neq 4\bmod 8 $ one has $\operatorname{O}(D(D_{m}))\cong \mathcal{C}_{2}$ and this group is generated by the permutation of the classes represented by $w$ and $\varepsilon_{1}+w$. This element is induced by $\sigma_{\varepsilon_{1}}\in\operatorname{O}(D_{m})$, the reflection at the hyperplane perpendicular to $\varepsilon_{1}$. Moreover $\operatorname{O}(D(D_{4}))\cong \mathcal{S}_{3}$. In this case the group is generated by the permutation of the classes $w$ and $\varepsilon_{1}+w$ and the permutation of $\varepsilon_{1}$ and $\varepsilon_{1}+w$. The latter element is induced by the reflection $\sigma_{w}\in\operatorname{O}(D_{4})$. Finally we note that the natural homomorphism \[\operatorname{O}(mA_{1})\to\operatorname{O}(D(mA_{1}))\]
is surjective. Summarizing these considerations we see that the assumption (\ref{Nikulins condition}) is satisfied for each $m=1,\dots,4$. This proves the first part and the invariance property of the pullbacks as a direct consequence.
\end{proof}
In the next step we construct a differential operator. This operator is well-known and a treatment can be found in \cite{ChoKi} for the general case or in \cite{EZ} for classical Jacobi forms.   
The heat operator is given as 
\[H=4\pi i \det (S) \frac{\partial}{\partial \tau}-\det (S)\,S^{-1}\left[\frac{\partial}{\partial \mathfrak{z}}\right]\,.\]
We recall the definition of the quasi-modular Eisenstein series of weight 2
\begin{equation}\label{G2 Fourier expansion}
 G_{2}(\tau)=-\frac{1}{24}+\sum_{n\geq 1}{\sigma_{1}(n)e^{2\pi in\tau }}\quad,\quad  \sigma_{k}(n)=\sum_{d\mid n}{d^{k}}
\end{equation}
which transforms under $\sl$ as
\[G_{2}\left(\frac{a\tau+b}{c\tau+d}\right)=(c\tau+d)^{2}G_{2}(\tau)-\frac{c(c\tau+d)}{4\pi i}\,.\]
and denote by $G_{2}\bullet$ the operator which multiplies a function by $G_{2}$. By virtue of the transformation property of $G_{2}$ we obtain a quasi-modular operator. We fix the notation
\[r^{l}:=\exp(2\pi i(l,\mathfrak{z}))\quad,\quad l\in L^{\vee},\mathfrak{z}\in L\otimes\C\,.\]
\begin{lemma}\label{lemma:heat operator with automorohic correction}
 For every $k\in\N$ there is a quasi-modular differential operator \newline $H_{k}:\,J_{k,L;1}\to J_{k+2,L;1}$ defined by the formula
\[\index{$H_{k}$}H_{k}=H+(4\pi i)^{2}\det(S)\left(k-\frac{m}{2}\right)G_{2}\bullet\]
where $m=\operatorname{rank}(L)$.  
The operator $H$ acts on $q^{n}r^{l}$, $n\in\N,l\in L^{\vee}$ by multiplication with \((2\pi i)^{2}\det(S)\,(2n-(l,l))\,.\)  
\end{lemma}
\begin{proof}
The first part can be deduced from the considerations in \cite{ChoKi} and the second part is a direct verification.
\end{proof}
Using this operator we define 
\[\index{$\epsilon_{6,mA_{1}}$}\epsilon_{6,mA_{1}}=H_{4}(\epsilon_{4,mA_{1}})\in J_{6,mA_{1};1}\,, m\in\{1,2,3,4\}\,.\]
These functions inherit the invariance under coordinate permutations from $\epsilon_{4,mA_{1}}$.
\begin{cor}\label{cor:differential operator preserves invariance for NA1}
 Let $m\in\{1,2,3,4\}$. For all $\sigma\in \operatorname{O}(mA_{1})$ we have 
\[\epsilon_{6,mA_{1}}(\tau,\sigma. \mathfrak{z})=\epsilon_{6,mA_{1}}(\tau,\mathfrak{z})\,.\] In particular $\epsilon_{6,mA_{1}}$ is invariant with respect to the action of $\operatorname{O}(D(mA_{1}))\cong \mathcal{S}_{m}$\,. 
\end{cor}
The next Lemma is about cusp forms and follows immediately from Lemma \ref{lemma:heat operator with automorohic correction}.
\begin{lemma}\label{differential operator and cusp forms}
 Let $\varphi\in J_{k,L;1}$. Then $\varphi$ is a cusp form if and only if $H_{k}(\varphi)$ is a cusp form. 
\end{lemma}
 Let $L\leq M$ be a primitive sublattice. We extend the pullback notation of Definition \ref{def:pullbak notation} to modular forms in the canonical way. In \cite{G} the arithmetic lifting of Jacobi forms has been defined. We denote this lifting operator by $\operatorname{A-Lift}(\cdot)$ and define 
\[\mathcal{E}_{k}^{mA_{1}}:=\operatorname{A-Lift}(\epsilon_{k,mA_{1}})\in\mathcal{M}_{k}(\widetilde{\Gamma}_{m},1)\,.\]
Moreover the operator $H_{k}$ is extended to the Maa{\ss} space by the following convention:
\[H_{k}(\operatorname{A-Lift}(\varphi)):=\operatorname{A-Lift}(H_{k}(\varphi))\quad,\quad \varphi\in J_{k,L;1}\,.\]
The following Theorem describes all modular forms obtained by the previous considerations.
\begin{theorem}\label{theorem: mod forms A1 tower}
We have the following diagram of modular forms for the $A_{1}$-tower
 \begin{center}
 \begin{tikzpicture}[scale=0.3,node distance=2cm,auto]
   \node (12+A1) at (0cm,-2.5cm) 	{$\textcolor{black}{F_{12}^{A_{1}}}$};
   \node (12+2A1) at (0cm,-6cm) {$\textcolor{black}{F_{12}^{2A_{1}}}$};
   \node (12+3A1) at (0cm,-9.5cm) {$\textcolor{black}{F_{12}^{3A_{1}}}$};
   \node (12+4A1) at (0cm,-13cm) {$F_{12}^{4A_{1}}$};
   \node[] (10+A1) at (4cm,-2.5cm) {$\textcolor{black}{G_{10}^{A_{1}}}$};
   \node[] (10+2A1) at (4cm,-6cm) {$\textcolor{black}{F_{10}^{2A_{1}}}$};
   \node[] (10+3A1) at (4cm,-9.5cm) {$\textcolor{black}{F_{10}^{3A_{1}}}$};
   \node[] (10+4A1) at (4cm,-13cm) {$F_{10}^{4A_{1}}$};
   \node[] (8+2A1) at (8cm,-6cm) {$\textcolor{black}{G_{8}^{2A_{1}}}$};
   \node[] (8+3A1) at (8cm,-9.5cm) {$\textcolor{black}{F_{8}^{3A_{1}}}$};
   \node[] (8+4A1) at (8cm,-13	cm) {$F_{8}^{4A_{1}}$};
   \node[] (6+3A1) at (12cm,-9.5cm) {$\textcolor{black}{G_{6}^{3A_{1}}}$};
  \node[] (6+4A1) at (12cm,-13cm) {$F_{6}^{4A_{1}}$};
  \node[] (4+4A1) at (16cm,-13cm) {$G_{4}^{4A_{1}}$};
  \node[] (e4+4A1) at (-9cm,-13cm) {$\mathcal{E}_{4}^{4A_{1}}$};
  \node[] (e4+3A1) at (-9cm,-9.5cm) {$\mathcal{E}_{4}^{3A_{1}}$};
  \node[] (e4+2A1) at (-9cm,-6cm) {$\mathcal{E}_{4}^{2A_{1}}$};
  \node[] (e4+A1) at (-9cm,-2.5cm) {$\mathcal{E}_{4}^{A_{1}}$};
  \node[] (e6+4A1) at (-4cm,-13cm) {$\mathcal{E}_{6}^{4A_{1}}$};
 \node[] (e6+3A1) at (-4cm,-9.5cm) {$\mathcal{E}_{6}^{3A_{1}}$};
 \node[] (e6+2A1) at (-4cm,-6cm) {$\mathcal{E}_{6}^{2A_{1}}$};
 \node[] (e6+A1) at (-4cm,-2.5cm) {$\mathcal{E}_{6}^{A_{1}}$};
\draw[<-](12+A1) to node {} (12+2A1);
\draw[<-](12+2A1) to node {} (12+3A1);
\draw[<-](12+3A1) to node {} (12+4A1);
\draw[<-](10+A1) to node {} (10+2A1);
\draw[<-](10+2A1) to node {} (10+3A1);
\draw[<-](10+3A1) to node {} (10+4A1);
\draw[<-](8+2A1) to node {} (8+3A1);In the next Theorem we 
\draw[<-](8+3A1) to node {} (8+4A1);
\draw[<-](6+3A1) to node {} (6+4A1);
\draw[->] (e4+4A1) to node {} (e4+3A1);
\draw[->] (e4+3A1) to node {} (e4+2A1);
\draw[->] (e4+2A1) to node {} (e4+A1);
\draw[]  (-2cm,-10.3cm) -- (14cm,-10.3cm);
\draw[]  (-2cm,-10.3cm) -- (-2cm,-1.5cm);
\draw[]  (14cm,-10.3cm) -- (14cm, -1.5cm);
\draw[]  (-2cm,-1.5cm) -- (14cm, -1.5cm);
\draw[->] (e4+A1) to node {\tiny$H_{4}$} (e6+A1);
\draw[->] (e4+2A1) to node {\tiny$H_{4}$} (e6+2A1);
\draw[->] (e4+3A1) to node {\tiny$H_{4}$} (e6+3A1);
\draw[->] (e4+4A1) to node {\tiny$H_{4}$} (e6+4A1);
\end{tikzpicture} 
\end{center}
where $F_{k}^{mA_{1}},G_{k}^{mA_{1}}\in\mathcal{M}_{k}(\Gamma_{m},1)$ for each form appearing in the diagram. The forms inside the rectangle are  cusp forms.
\end{theorem}
\begin{proof}
 We define the $F_{k}^{mA_{1}},G_{k}^{mA_{1}}$ as the arithmetic liftings of the functions in Proposition \ref{prop:tower for 4A1} with the same arrangement for the weights and the lattices. This yields functions which belong to $\mathcal{M}_{k}(\widetilde{\Gamma}_{m},1)$. The arithmetic lifting maps cusp forms to cusp forms if the lattice $mA_{1}$ is a maximal even lattice. This is precisely the case for $m=1,2,3$. Hence the statements on cusp forms follow from Proposition \ref{prop:tower for 4A1} and the preceeding construction of $\epsilon_{k,mA_{1}}$. Note that the operator $\operatorname{A-Lift}$ commutes with pullbacks as one immediately extracts from its definition. 
 The Jacobi forms appearing in Proposition \ref{prop:modularity of the E8 pullbacks for A1 tower} and Corollary \ref{cor:differential operator preserves invariance for NA1} are invariant with respect to the permutations of $z_{1},\dots,z_{m}$. However, the theta type forms listed in Proposition \ref{prop:tower for 4A1} are not except for the functions $\psi_{12-2m,mA_{1}}$ on the diagonal. Hence we can apply the operator
\begin{equation}\label{symmetrization operator}
J_{k,mA_{1};1}\to J_{k,mA_{1};1}\quad,\quad\varphi\mapsto\frac{1}{m!}\sum_{\sigma\in \mathcal{S}_{m}}{\sigma.\varphi}
\end{equation}
where $(\sigma.\varphi)(\tau,\mathfrak{z}):=\varphi(\tau,\sigma.\mathfrak{z})$  for any $\sigma\in\mathcal{S}_{m}$. After application of (\ref{symmetrization operator}) all the functions appearing in Proposition \ref{prop:tower for 4A1} are symmetric. Hence the maximal modular group of these liftings is  $\Gamma_{m}$. 
\end{proof}

\section{Rings of Modular Forms}
For any $r\in L_{2}\otimes \Q$ satisfying $(r,r)<0$ we define the rational quadratic divisor as
\[\mathcal{D}_{r}=\{[\mathcal{Z}]\in\mathcal{D}\,|\,(\mathcal{Z},r)=0\}\,.\]
For any $m\in\N$ we fix the notation
\[\mathcal{D}^{m}:=\mathcal{D}(L_{2}(mA_{1}))\,.\]
In particular consider $\varepsilon_{m}\in mA_{1}(-1)\subseteq L_{2}(mA_{1})$. Then one has
\begin{equation}\label{rational quadratc divisors for A1 tower}
 \mathcal{D}_{\varepsilon_{m}}^{m}\cong \mathcal{D}^{m-1}\,.
\end{equation}
 In \cite[Theorem 5.1]{G1} it was proved that this divisor is attached to $G_{12-2m}^{mA_{1}}$ for $m=1,\dots,4$.
\begin{theorem}\label{Lifting construction of reflective modular forms for A1 tower}
 Let $m\in\{1,2,3,4\}$. The divisor of the modular form $G_{12-2m}^{mA_{1}}$ consists of the $\Gamma_{m}$-orbit of
\(\mathcal{D}_{\varepsilon_{m}}^{m}.\)
The vanishing order is two on each irreducible component of $\operatorname{div}(G_{12-2m}^{mA_{1}})$. Moreover there exists a modular form 
\[\chi_{6-m}\in\mathcal{M}_{6-m}(\Gamma_{m},\det v_{2})\] 
whose square equals $G_{12-2m}^{mA_{1}}$. If $m\neq 4$, then $\chi_{6-m}$ is a cusp form.
\end{theorem}
This yields the structure of the graded ring for the $A_{1}$-tower with trivial character.
\begin{theorem}\label{theorem:graded ring with trivial character}
 Let $m\in\{1,2,3,4\}$. The graded ring $\mathcal{A}(\Gamma_{m})$ is a polynomial ring in the $m+3$ functions which are given by the $m$-th row of the diagram in Theorem \ref{theorem: mod forms A1 tower}. 
\end{theorem}
\begin{proof}
 Since $-I_{m+4}\in\Gamma_{m}$ for all $m$ there are no modular forms of odd weight in $\mathcal{A}(\Gamma_{m})$. For the proof we consider the following reduction process:
\begin{enumerate}[(i)]
 \item The starting point is the case $m=1$. Our construction yields the classical result of Igusa
\begin{equation}\label{Igusas result for A1}
 \mathcal{A}(\Gamma_{1})=\C[\mathcal{E}_{4}^{A_{1}},\mathcal{E}_{6}^{A_{1}},G_{10}^{A_{1}},F_{12}^{A_{1}}]
\end{equation}
as we have already investigated in (\ref{Igusas result}).
 \item We have an embedding $\Gamma_{m}\hookrightarrow \Gamma_{m+1}$ for each $m\in\N$. Hence the restriction map  
\[\mathcal{M}_{k}(\Gamma_{m+1},1)\to\mathcal{M}_{k}(\Gamma_{m},1)\,,\,F\mapsto F|_{\mathcal{D}^{m}}\]
 is well-defined for any even $k\in\N_{0}$. This map extends to a homomorphism of the graded algebras
\[\operatorname{Res}^{m+1}_{m}\,:\,\mathcal{A}(\Gamma_{m+1})\to \mathcal{A}(\Gamma_{m})\,.\]
We shall show that this map is surjective for $m=1,2,3$. 
\item Let $k\in\N_{0}$ and consider $F\in\mathcal{M}_{k}(\Gamma_{m},1)$ with the property \[F|_{\mathcal{D}^{m-1}}\equiv 0\,.\] Let $m\geq 2$ and define
\[M:=\operatorname{diag}(1,1,K,1,1)\,\text{ where }\,K:=\operatorname{diag}(1,\dots,-1)\in\operatorname{GL}(m,\Z)\,.\]
Since $M$ belongs to  $\Gamma_{m}$ the Taylor expansion of $F$ around $0$ with respect to $z_{m}$ shows that $F$ vanishes of order at least two on $\mathcal{D}^{m-1}$. According to Theorem \ref{Lifting construction of reflective modular forms for A1 tower}
we can divide $F$ by $G_{12-2m}^{mA_{1}}$ and obtain a holomorphic modular form in $\mathcal{M}_{k-12+2m}(\Gamma_{m},1)$
by Koecher's principle for automorphic forms, compare \cite[p. 209]{Ba}. Note that we have used the identification (\ref{rational quadratc divisors for A1 tower}), here.
\end{enumerate}
Now starting from (\ref{Igusas result for A1}) the statement follows by induction on the weight using (i)-(iii) where the surjectivity of $\operatorname{Res}_{m}^{m+1}$ for $m=1,2,3$ is extracted from Theorem \ref{theorem: mod forms A1 tower}.
\end{proof}
In the following we construct three modular form with respect to the character $v_{\pi}$. Following  \cite{Wo1} we consider the following three theta type Jacobi forms with respect to the coordinates introduced in (\ref{choice of coordinates for mA1})
\[\begin{aligned}
   &\vartheta_{4A_{1}}^{(1)}(\tau,\mathfrak{z}_{4A_{1}})=\vartheta(\tau,z_{1}-z_{2})\vartheta(\tau,z_{1}+z_{2})\vartheta(\tau,z_{3}-z_{4})\vartheta(\tau,z_{3}+z_{4})\,,&\\
   &\vartheta_{4A_{1}}^{(2)}(\tau,\mathfrak{z}_{4A_{1}})=\vartheta(\tau,z_{3}-z_{2})\vartheta(\tau,z_{3}+z_{2})\vartheta(\tau,z_{1}-z_{4})\vartheta(\tau,z_{1}+z_{4})\,,&\\
   &\vartheta_{4A_{1}}^{(3)}(\tau,\mathfrak{z}_{4A_{1}})=\vartheta(\tau,z_{1}-z_{3})\vartheta(\tau,z_{1}+z_{3})\vartheta(\tau,z_{2}-z_{4})\vartheta(\tau,z_{2}+z_{4})&
  \end{aligned}\] 
such that $\vartheta_{4A_{1}}^{(j)}\in J_{2,4A_{1};1}(v_{\eta}^{12})$ for $j=1,2,3$. By multiplying each of the three functions by $\eta(\tau)^{12}$ and considering the arithmetic lifting of these functions we obtain three modular forms $\Delta_{8,4A_{1}}^{(j)},j=1,2,3$ which belong to   $\mathcal{M}_{8}(\widetilde{\Gamma}_{4},1)$. 
\begin{prop}\label{cusp form of weight 24 for 4A1}
 There is a cusp form \[\Delta_{24}^{4A_{1}}\in\mathcal{S}_{24}(\Gamma_{4},v_{\pi})\] satisfying 
\[\Delta_{24}^{4A_{1}}=\Delta_{8,4A_{1}}^{(1)}\,\Delta_{8,4A_{1}}^{(2)}\,\Delta_{8,4A_{1}}^{(3)}\,.\]
The divisor equals the $\Gamma_{4}$-orbit of $\mathcal{D}_{\varepsilon_{1}+\varepsilon_{4}}^{4}$.
\end{prop}
\begin{proof}
 This function coincides with the cusp form of weight 24 for the lattice $D_{4}$ which was constructed in \cite[Theorem 4.4]{Wo1}.  Since $4A_{1}$ is a sublattice of $D_{4}$ we obtain a function with the modular behaviour stated above where the character $v_{\pi}$ appears due to the definition of $\vartheta_{4A_{1}}^{(j)}$ above. For a maximal even lattice the arithmetic lifting of a Jacobi form is a cusp form if the Fourier expansion ranges over all parameters with positive hyperbolic norm, compare \cite[Theorem 3.1]{G}. This characterization can be extended to all lattices with the property that every isotropic subgroup of $D(L_{2})$ is cyclic, see \cite[Theorem 4.2]{GHS1} for a proof. Since this is the case for the lattice $L_{2}(4A_{1})$ the function $\Delta_{24}^{4A_{1}}$ is a cusp form.
\end{proof}
The Proposition yields two more cusp forms for the tower.
\begin{cor}\label{cor:construction of anti symmetric modular form for 3A1 of weight 18}
 There are  cusp forms 
\[\Delta_{10}^{2A_{1}}\in\mathcal{S}_{10}(\Gamma_{2},v_{\pi})\quad\text{ and }\quad \Delta_{18}^{3A_{1}}\in\mathcal{S}_{18}(\Gamma_{3},v_{\pi})\] 
whose divisor equals 
the $\Gamma_{m}$-orbit of 
\[\mathcal{D}_{\varepsilon_{1}+\varepsilon_{m}}^{m}\quad, \quad m=2\text{ or }3\text{, respectively.}\]
\end{cor}
\begin{proof}
 We consider the cusp form $\Delta_{24}^{4A_{1}}$ in Proposition \ref{cusp form of weight 24 for 4A1}. The divisor is the sum of the six $\widetilde{\Gamma}_{m}$-orbits which are represented by $\mathcal{D}_{1}^{4}+\mathcal{D}_{2}^{4}$ where
\[\begin{aligned}
   &\mathcal{D}_{1}^{4}:=\mathcal{D}_{\varepsilon_{1}+\varepsilon_{2}}^{4}+\mathcal{D}_{\varepsilon_{1}+\varepsilon_{3}}^{4}+\mathcal{D}_{\varepsilon_{2}+\varepsilon_{3}}^{4}\,,&\\
&\mathcal{D}_{2}^{4}:=\mathcal{D}_{\varepsilon_{1}+\varepsilon_{4}}^{4}+\mathcal{D}_{\varepsilon_{2}+\varepsilon_{4}}^{4}+\mathcal{D}_{\varepsilon_{3}+\varepsilon_{4}}^{4}\,.&
  \end{aligned}\] 
The group $\operatorname{O}(D(4A_{1}))\cong \mathcal{S}_{4}$ acts $2$-fold transitive on the set 
\[\{\varepsilon_{\mu}\,|\,1\leq \mu \leq 4\}\]
by relabelling the indices. The group $\mathcal{S}_{4}$ contains $\mathcal{S}_{3}$ as the subgroup fixing $\varepsilon_{4}$. 
The sets $\mathcal{D}_{\varepsilon_{k}+\varepsilon_{4}}^{4},\mathcal{D}_{\varepsilon_{k}-\varepsilon_{4}}^{4}$ where $k=1,2,3$ belong to the same $\widetilde{\Gamma}_{4}$-orbit but constitute different $\widetilde{\Gamma}_{3}$-orbits.  Hence the restriction of $\Delta_{24}^{4A_{1}}$ to $\mathcal{D}^{3}$ has the divisor
\[\mathcal{D}_{1}^{3}+2\,\mathcal{D}_{2}^{3}\]
with respect to the action of $\widetilde{\Gamma}_{3}$ where
\[\mathcal{D}_{j}^{3}:=\mathcal{D}_{j}^{4}\cap \mathcal{D}^{3}\,,\,j=1,2\,.\]
Moreover the $\Gamma_{3}$-orbit of the restriction is represented by 
\[\mathcal{D}_{\varepsilon_{1}+\varepsilon_{3}}^{3}+2\,\mathcal{D}_{\varepsilon_{3}}^{3}\,.\]
According to Theorem \ref{Lifting construction of reflective modular forms for A1 tower} we can define
\[\Delta_{18}^{3A_{1}}:=\frac{\Delta_{24}^{4A_{1}}\Big|_{\mathcal{D}^{3}}}{G_{6}^{3A_{1}}}\]
and obtain a modular form with the desired poroperties by Koecher's principle. Now the same construction is done with $\Delta_{18}^{3A_{1}}$ instead and one defines
\[\Delta_{10}^{2A_{1}}:=\frac{\Delta_{18}^{3A_{1}}\Big|_{\mathcal{D}^{2}}}{G_{8}^{2A_{1}}}\]
which has the correct divisor. An analysis of the Fourier expansion of $\Delta_{10}^{2A_{1}},\Delta_{18}^{3A_{1}}$ yields that both functions are cusp forms.
\end{proof}
In the following we consider another type of modular forms. Let $II_{2,26}$ be the unique (up to isomorphism) even unimodular lattice of signature $(2,26)$. We define 
\[\index{$R_{-2}^{II_{2,26}}$}R_{-2}^{II_{2,26}}:=\{r\in II_{2,26}\,|\,(r,r)=-2\}\,.\]
The following statement is due to Borcherds and can be found in \cite[Theorem 10.1 and Example 2]{Bo2}.
\begin{theorem}[Borcherds]\label{theorem:construction of Borcherds Phi12}
 There is a holomorphic modular form $\Phi_{12}$ with the properties \index{$\Phi_{12}$}\[\Phi_{12}\in\mathcal{M}_{12}(\operatorname{O}(II_{2,26})^{+},\det)\quad, \quad \operatorname{div}(\Phi_{12}) =\bigcup_{r\in R^{II_{2,26}}_{-2}}{\mathcal{D}_{r}(II_{2,26})}\] 
where the vanishing order is exactly one on each irreducible component.
\end{theorem}
In \cite[Example 2]{Bo2} Borcherds computes the Fourier expansion of $\Phi_{12}$. It turns out that $\Phi_{12}$ reflects the Weyl denominator formula for the fake monster Lie algebra. 
\vspace{3mm}

\noindent Let $\mathcal{N}$ be the Niemeier lattice with root system $24A_{1}$, see \cite{Nie}. Since \[II_{2,26}\cong U \perp U_{1}\perp \mathcal{N}\] where 
\[U=\langle e,f\rangle,U_{1}=\langle e_{1},f_{1}\rangle\] 
are two integral hyperbolic planes we can consider the natural embedding 
\begin{equation}\label{chioce of embedding into II 2,26}
 L_{2}(mA_{1})\hookrightarrow U \perp U_{1}\perp \mathcal{N} \quad,\quad m\in\{1,2,3,4\} 
\end{equation}
which is induced by $mA_{1}\hookrightarrow 24A_{1}$. 
Let $K_{m}$ be the orthogonal complement of $L_{2}(mA_{1})$ in $II_{2,26}$. Each vector $r\in II_{2,26}$ has a unique decomposition
\[r=\alpha(r)+\beta(r)\quad,\quad \alpha(r)\in L_{2}(mA_{1})^{\vee}\,,\,\beta(r)\in K_{m}^{\vee}\,.\]
 We set
\[\index{$R_{-2}(K)$}R_{-2}(K_{m})=\{r\in II_{2,26}\,|\,(r,r)=-2\,,\, r\perp L_{2}(mA_{1}) \},\]
which is contained in the negative definite lattice $K_{m}$ and hence finite. Consequently we define \index{$\operatorname{N}(K)$}$N(K_{m})=\frac{\sharp R_{-2}(K_{m})}{2}\in\N\,.$
The next statement is a special case of \cite[Theorem 8.2 and Corollary 8.12]{GHS2} and describes the construction of a quasi-pullback from Borcherds function $\Phi_{12}$.
\begin{theorem}\label{theorem:quasi pullback form phi12}
Consider a primitive embedding $L_{2}\hookrightarrow II_{2,26}$ and denote by $K$ the orthogonal complement of $L_{2}$ in $II_{2,26}$.
\[\left.\Phi\right|^{\textit{(QP)}}_{L_{2}}(\mathcal{Z})=\displaystyle\left.\frac{\Phi_{12}(\mathcal{Z})}{\prod_{r\in R_{-2}(K)/{\{\pm 1\}}}{(\mathcal{Z},r)}}\right|_{\mathcal{D}(L_{2})},\]
where in the product one fixes a set of representatives for $R_{-2}(K)/\{\pm 1\}$. Then  $\left.\Phi\right|^{\textit{(QP)}}_{L_{2}}$ belongs to $\mathcal{M}_{12+N(K)}(\widetilde{\operatorname{O}}(L_{2})^{+},\det)$
and vanishes exactly on all rational quadratic divisors  
\[\mathcal{D}_{\alpha(r)}=\{[\mathcal{Z}]\in\mathcal{D}(L_{2})\,|\,(\mathcal{Z},\alpha(r))=0\}\] 
where $r$ runs through the set \(R_{-2}^{II_{2,26}}\) and $(\alpha(r),\alpha(r))<0$. If $N(K)>0$ we say that $\left.\Phi\right|^{\textit{(QP)}}_{L_{2}}$ is a quasi-pullback of $\Phi_{12}$\index{quasi-pullback of $\Phi_{12}$}. In this case $\left.\Phi\right|^{\textit{(QP)}}_{L_{2}}$ is a cusp form. 
 \end{theorem}
The choice of the embedding (\ref{chioce of embedding into II 2,26}) yields another four modular forms with respect to a character. 
\begin{theorem}\label{baby monster type modular forms for A1 tower}
Let $m\in\{1,2,3,4\}$. There exists a cusp form 
\[F_{36-m}^{mA_{1}}\in \mathcal{S}_{36-m}(\Gamma_{m},\det v_{\pi}^{\kappa})\quad ,\quad \kappa\in\{0,1\}\] 
whose divisor is represented by the sum of the two different $\Gamma_{m}$-orbits
\[\mathcal{D}_{e_{1}-f_{1}}^{m}+\mathcal{D}_{\varepsilon_{m}}^{m}\,.\]
\end{theorem}
\begin{proof}
 The results can be extracted from \cite{Gr}. However, we give a sketch of the proof, here, using Theorem \ref{theorem:quasi pullback form phi12}. The strategy of the proof is to construct modular forms as quasi-pullbacks of $\Phi_{12}$ with respect to the embedding (\ref{chioce of embedding into II 2,26}). We denote these forms by $F_{k_{m}}^{mA_{1}}$ for $m=1,2,3,4$. Since the number of $-2$-roots for the lattice $mA_{1}$ is exactly $2m$ we have $N(K_{m})=24-m$ and the weight $k_{m}$ of $F_{k_{m}}^{mA_{1}}$ is exactly $12+N(K_{m})=36-m$. The divisor of $F_{36-m}^{mA_{1}}$ is determined by all vectors $\alpha\in L_{2}(mA_{1})^{\vee},(\alpha,\alpha)< 0$ such that there exists a $\beta\in K_{m}^{\vee}$ satisfying $\alpha+\beta\in R_{-2}^{II_{2,26}}$. The choice of our embedding already implies $(\alpha,\alpha)=-2$ and $\beta(r)=0$. There are $m+1$ different orbits of $-2$-roots in $L_{2}(mA_{1})$ with respect to  $\widetilde{\Gamma}_{m}$ represented by 
 \[e_{1}-f_{1},\varepsilon_{1},\dots,\varepsilon_{m}\,.\]
The divisor is represented by the sum of these orbits. In \cite[Corollary 15.2]{N} the author gave a criterion to decide whether an element $g\in L_{2}(mA_{1})$ extends to $\operatorname{O}(II_{2,26})^{+}$. In our case  this is always possible, compare \cite[p. 122]{Gr}. Hence the maximal modular group is indeed larger and the divisor with respect to the larger group is represented by the two orbits stated above.
\end{proof}
We define the four functions 
\[H_{30}^{mA_{1}}:=\frac{F_{36-m}^{mA_{1}}}{\chi_{6-m}^{mA_{1}}}\quad \text{ where }\,m\in\{1,2,3,4\}\]
whose divisor is represented by the $\Gamma_{m}$-orbit of
\(\mathcal{D}_{e_{1}-f_{1}}^{m}.\)
The next Lemma 	is useful in order to determine the graded ring $\mathcal{A}(\Gamma'_{m})$.
\begin{lemma}\label{lemma:divisors of A1 modular forms with characer}
  Let $F\in\mathcal{M}_{k}(\Gamma_{m},\lambda)$ for some finite character $\lambda:\Gamma_{m}\to\C^{\ast}$ and $m=1,2,3,4$.
\begin{enumerate}[(i)]
 \item If $\lambda=\det v_{2}^{a} v_{\pi}^{b}$ where $a,b\in\{0,1\}$ then $\mathcal{D}_{\varepsilon_{m}}^{m}\leq \operatorname{div}(F)$. 
 \item If $\lambda=v_{\pi}v_{2}^{a}$ where $a\in\{0,1\}$  then $\mathcal{D}_{\varepsilon_{1}-\varepsilon_{m}}^{m}\leq \operatorname{div}(F)$. 
 \item If $\lambda=v_{2}v_{\pi}^{a}$ where $a\in\{0,1\}$ then $\mathcal{D}_{e_{1}-f_{1}}^{m}\leq \operatorname{div}(F)$. 
\end{enumerate}
 \end{lemma}
\begin{proof}
\begin{enumerate}[(i)]
  \item The reflection $\sigma_{\varepsilon_{m}}$  satisfies $\lambda(\sigma_{\varepsilon_{m}})=-1$. Hence its fixed locus $\mathcal{D}_{\varepsilon_{m}}^{m}$ is contained in $\operatorname{div}(F)$. 
 \item We have $\lambda(\sigma_{\varepsilon_{j}-\varepsilon_{l}})=-1$ where $j,l$ are distinct numbers in $\{1,\dots,m\}$ and this reflection induces the permutation $z_{j}\mapsto z_{k},z_{k}\mapsto z_{j}$ with respect to our standard basis (\ref{choice of coordinates for mA1}). This shows that the fixed locus of this reflection is part of $\operatorname{div}(F)$.
 \item We consider the reflection $\sigma_{1}=\sigma_{e_{1}-f_{1}}$.  From \cite[Section 1.6.3 ]{Kl} and the formula for $P$ in the proof of \cite[Corollary 1.23]{Kl} we deduce that $v_{2}(\sigma_{1})=-1$. As $\mathcal{D}_{e_{1}-f_{1}}^{m}$ is exactly the fixed locus of this reflection we are done.
\end{enumerate}
\end{proof}
Now we are able to prove the main theorem by extending the diagram given in Theorem \ref{theorem: mod forms A1 tower}.
\begin{mainproof}
 Let $F\in \mathcal{A}(\Gamma'_{m})$. Without loss of generality we can assume that $F$ is homogeneous of weight $k$. If the character of $F$ is trivial the assertion follows from Theorem \ref{theorem:graded ring with trivial character}. In the general situation we can use Lemma \ref{lemma:divisors of A1 modular forms with characer} to divide $F$ by one of the forms $H_{30}^{mA_{1}},\chi_{6-m}^{mA_{1}}$ or $\Delta_{k_{m}}^{mA_{1}}$ in the case $m\neq 1$. By virtue of Koecher's principle this process yields a modular form with trivial character and we are back in the first case. 
\end{mainproof}
Let $m\in\{2,3,4\}$.  The only relations among the generators of $\mathcal{A}(\Gamma'_{m})$ are given by 
\[(\Delta_{k_{m}}^{mA_{1}})^{2}=P_{mA_{1}}(\mathcal{E}_{4}^{mA_{1}}, \mathcal{E}_{6}^{mA_{1}}, \chi_{6-m}^{mA_{1}}, F_{14-2m}^{mA_{1}},\dots, F_{12}^{mA_{1}})\]
and
\[(H_{30}^{mA_{1}})^{2}=Q_{mA_{1}}(\mathcal{E}_{4}^{mA_{1}}, \mathcal{E}_{6}^{mA_{1}}, \chi_{6-m}^{mA_{1}}, F_{14-2m}^{mA_{1}},\dots, F_{12}^{mA_{1}})\]
where $P_{mA_{1}},Q_{mA_{1}}\in \C[X_{1},\dots,X_{m+3}]$ are uniquely determined  polynomials. In the case $m=1$ our methods yield Igusa's description.

\section{Eisenstein series}
Let $L$ be a positive definite even lattice. In \cite{BrK} the authors investigated vector valued Eisenstein series. We denote the vector valued Eisenstein series of weight $k$ with respect to the simplest cusp by $\vec{e}_{L,k}$, compare \cite[p. 23]{Br}. We denote by $(\mathfrak{e}_{\gamma})_{\gamma\in D(L)}$ the standard basis for the group algebra $\C[L^{\vee}/L]$. For $k\geq 5/2,k\in \Z/2$ these functions are vector valued modular forms with respect to the dual of the Weil representation for $L$ with Fourier expansion
\[\vec{e}_{L,k}(\tau)=\sum_{\gamma\in D(L)}\sum_{\substack{n\in \Z-\frac{1}{2}(\gamma,\gamma)\\n \geq 0}}{c(n,\gamma,L,k)\exp\left(2\pi i n\tau\right)\mathfrak{e}_{\gamma}}\quad,\text{ where }\tau\in \mathbb{H}\,.\] 
 In \cite[Theorem 4.6]{BrK} an explicit formula for $c(n,\gamma,L,k)$ is given. 
There is a well-known correspondence between vector valued modular forms and Jacobi forms for the lattice $L$, compare \cite[Proposition 1.6]{Sh}. Let $m\in\N,m\leq 8$. We define
\[s_{m}(n,l):=\sharp\{x\in ((mA_{1})_{E_{8}}^{\perp})^{\vee}\,|\,l+x\in E_{8}\,,\,(l+x,l+x)=2n\}\in\Z\,.\]
Note that these numbers are exactly the Fourier coefficients of $\epsilon_{4,mA_{1}}$, more precisely one has
\[\epsilon_{4,mA_{1}}(\tau,\mathfrak{z})=\sum_{\substack{n\in\Z_{\geq 0}\\l\in mA_{1}^{\vee}\\2n-(l,l)\geq 0}}{s_{m}(n,l)q^{n}r^{l}}\]
Since the lattice $E_{8}$ is a maximal even lattice there are no nontrivial isotropic subgroups of $D(E_{8})$. Hence we can rewrite this quantity in the simplified form
\begin{equation}\label{representation numbers of A1 tower in E8}
 s_{m}(n,l)=\sharp\{x\in ((mA_{1})_{E_{8}}^{\perp})^{\vee}\,|\,(l+x,l+x)=2n\}\quad,\quad n\in \N\,,\, l\in (mA_{1})^{\vee}.
\end{equation}
We describe the numerical values of the Eisenstein-like Jacobi forms if $J_{k,mA_{1};1}^{(\textit{cusp})}=\{0\}$ and $m\leq 3$.
\begin{prop}\label{prop:Fourier expansion of Eisenstein type}
\begin{enumerate}[(a)]
 \item  Let $m=1,2,3$. For any  $n\in\Z_{\geq 0},l\in (mA_{1})^{\vee}$ we have the identity  
\[2s_{m}(n,l)=c\left(n-(l,l)/2,l \bmod L,mA_{1},4-m/2\right)\,.\]
The first Fourier coefficients of $\epsilon_{4,mA_{1}}$ are given as follows:
\[\begin{aligned}
   &\varepsilon_{4,A_{1}}(\tau,\mathfrak{z})=&&1+(r^{-1}+56r^{-1/2}+126+56r^{1/2}+r)q&\\&&&+(126r^{-1}+576r^{-1/2}+756+576r^{1/2}+126r)q^{2}&\\&&&+(56r^{-3/2}+756r^{-1}+1512 r^{-1/2}+2072+1512r^{1/2}+756r+56r^{3/2})q^{3}&\\&&&+(\dots) q^{4}\,,&\\\\
   &\varepsilon_{4,2A_{1}}(\tau,\mathfrak{z})=&&1+(60+32r^{(\pm 1/2,0)}+32r^{(0,\pm 1/2)}+12r^{(\pm 1/2,\pm 1/2)}+r^{(\pm 1,0)}+r^{(0,\pm 1)})q&\\
   &&&+(252+192r^{(\pm 1/2,0)}+192r^{(0,\pm 1/2)}+160r^{(\pm 1/2, \pm 1/2)}&\\&&&+60r^{(\pm 1 ,0)}+60r^{(0, \pm 1)}+32r^{(\pm 1,\pm 1/2)}+32r^{(\pm 1/2, \pm 1)}+r^{(\pm 1, \pm 1)})q^{2}&\\&&&+(\dots)q^{3}\,,&\\\\
   &\varepsilon_{4,3A_{1}}(\tau,\mathfrak{z})=&&1+(26+16r^{(\pm 1/2,0,0)}+16r^{(0,\pm 1/2 ,0)}+16r^{(0,0,\pm 1/2)}+8r^{(\pm 1/2,\pm 1/2,0)}&\\&&&+8r^{(\pm 1/2,0,\pm 1/2)}+8r^{(0,\pm 1/2,\pm 1/2)}+2r^{(\pm 1/2,\pm 1/2,\pm 1/2)}&\\&&&+r^{(\pm 1,0,0)}+r^{(0,\pm 1,0)}+r^{(0,0,\pm 1)})q&\\&&&+(\dots)q^{2}
  \end{aligned}\]
 \item Let $m=1,2$ and define the numbers $r_{m}(n,l)$ by
\[\epsilon_{6,mA_{1}}(\tau,\mathfrak{z})=\sum_{\substack{n\in\Z_{\geq 0}\\l\in mA_{1}^{\vee}\\2n-(l,l)\geq 0}}{r_{m}(n,l)q^{n}r^{l}}\,.\]
For any $n\in\Z_{\geq 0},l\in (mA_{1})^{\vee}$ we have
\[\frac{3\cdot 2^{1-m}}{16-2m}\cdot r_{m}(n,l)\in\Z\pi^{2}\,.\]
\end{enumerate}
\end{prop}
\begin{proof}
We consider the subspace of $J_{k,mA_{1},1}$ consisting of all functions which belong to the kernel of the symmetrization operator (\ref{symmetrization operator}). Due to Theorem \ref{theorem:graded ring with trivial character} this space is one-dimensional if  $(m,k)$ is contained in the set
\[\Omega=\{(1,4),(2,4),(3,4),(1,6),(2,6)\}\,.\]
In each case the space is generated by $\epsilon_{k,mA_{1}}$. Any zero-dimensional cusp of the modular variety $\Gamma_{m}\backslash\mathcal{D}^{m}$ is represented by a primitive isotropic vector in $L_{2}(mA_{1})$. Two zero-dimensional cusps are equivalent if they belong to the same $\Gamma_{m}$-orbit. In the cases where $m< 4$ all zero-dimensional cusps are equivalent to the simplest cusp. In this case the codimension of the subspace $J_{k,mA_{1};1}^{(\textit{cusp})}\subseteq J_{k,mA_{1};1}$ is one. In \cite{Br} the author defined Eisenstein series for every zero-dimensional cusp represented by a vector $c\in L^{\vee}$ such that $(c,c)\in 2\Z$. The corresponding Jacobi forms are denoted by $e_{k,c}^{L}$. The span of these functions, as $c$ runs through all zero-dimensional cusps, constitutes the complementary space of $J_{k,mA_{1};1}^{(\textit{cusp})}$.  We call the members of this space Jacobi-Eisenstein series. They can be viewed as the natural generalization of the classical Jacobi-Eisenstein series investigated in \cite{EZ}. For any permutation $\sigma\in \mathcal{S}_{m}$ the formula for the Fourier coefficients of $e_{k,c}^{mA_{1}}$ yields
\[\sigma.e_{k,c}^{mA_{1}}=e_{k,\sigma c}^{mA_{1}}\,.\]
Since $\sigma c$ is equivalent to $c$ this identity shows that $e_{k,c}^{mA_{1}}$ belongs to the kernel of the operator (\ref{symmetrization operator}) for any $c$ satisfying $(c,c)\in2\Z$. Choosing $c=0$ we infer that  for any $(m,k)\in\Omega$ there exists some $\lambda\in\C^{\ast}$ such that 
\[\epsilon_{k,mA_{1}}=\lambda\, e_{k,0}^{mA_{1}}\,.\]
Using \cite[Theorem 4.6]{BrK} we can express the Fourier coefficients of $\epsilon_{k,mA_{1}}$ as special values of $L$-functions.
 \begin{enumerate}[(a)]
  \item A comparison of the Fourier coefficients yields $\lambda=\frac{1}{2}$ if $k=4$ and we obtain the numerical values of $\epsilon_{4,mA_{1}}$ by evaluating the formula for the Fourier coefficients of $e_{k,0}^{mA_{1}}$.
\item The Fourier coefficient of the constant term of $\epsilon_{6,A_{1}}$ equals 
\[-(2\pi i)^{2}\det(A_{1})\frac{14}{24}=\frac{14}{3}\pi^{2}\,.\]
The Fourier coefficients of $\frac{1}{2}e_{6,0}^{A_{1}}$ are well-known to be integral, see \cite[Theorem I.2.1]{EZ} and \cite{Coh}. Now the statement in the case $m=1$ follows by a comparison of the constant term. In the case $m=2$ the constant term of $\epsilon_{6,2A_{1}}$ is given by
\[-(2\pi i)^{2}\det(2A_{1})\frac{12}{24}=8\pi^{2}\,.\]
From Lemma \ref{lemma:heat operator with automorohic correction} we obtain the Fourier expansion of $\epsilon_{6,2A_{1}}(\tau,\mathfrak{z})$ as
\[(8\pi^{2})\cdot 24\,G_{2}(\tau)\cdot\epsilon_{4,2A_{1}}(\tau,\mathfrak{z})-\sum_{\substack{n\in\Z_{\geq 0}\\l\in (2A_{1})^{\vee}\\2n-(l,l)\geq 0}}{\alpha(n,l)q^{n}r^{l}}\]
where
\[\alpha(n,l):=(8\pi^{2})\cdot (4n-2(l,l))\, s_{m}(n,l)\,.\]
Since $(4n-2(l,l))\in\Z$ for all $n\in\Z_{\geq 0},l\in (2A_{1})^{\vee}$ we obtain the assertion as a consequence of formula (\ref{G2 Fourier expansion}) and (\ref{representation numbers of A1 tower in E8}).
\end{enumerate}
\end{proof}
The last Proposition jusitifies the notion Eisenstein type. In the cases considered there the space complementary to the space of cusp forms is always one-dimensional. If $m=4$ there is no Jacobi-Eisenstein series to be considered since $4-m/2=2<5/2$. In this case $\epsilon_{4,4A_{1}}$ is the correct replacement for the missing Eisenstein-series. Moreover there are two inequivalent zero-dimensional cusps. Here the complementary space is two-dimensional. If $k=4$ the first generator of this space is given by the Eisenstein type modular form $\mathcal{E}_{4}^{A_{1}}$ and the second generator coincides with the square of $\chi_{2}^{4A_{1}}$ which is the main function of the $A_{1}$-tower. By virtue of the identities
\[\begin{aligned}
   &\sharp\{l\in E_{7}\,|\,(l,l)=2n\}&&=&&c(n,0,A_{1},7/2)/2&\\
   &\sharp\{l\in D_{6}\,|\,(l,l)=2n\}&&=&&c(n,0,2A_{1},3)/2&\\
   &\sharp\{l\in A_{1}\oplus D_{4}\,|\,(l,l)=2n\}&&=&&c(n,0,3A_{1},5/2)/2&
  \end{aligned}\]
  Proposition \ref{prop:Fourier expansion of Eisenstein type} yields a new description of these representation numbers of quadratic forms as special values of $L$-functions. 
\bibliography{/home/martin/Bibliography/bibliography}
\bibliographystyle{plain}
\end{document}